\documentclass{article}
\usepackage{amsfonts,amssymb,amsmath,amsthm}
\usepackage{chngpage}
\newtheorem{example}{Example}[section]
\newtheorem{theorem}[example]{Theorem}
\newtheorem{proposition}[example]{Proposition}
\newtheorem{definition}[example]{Definition}
\newtheorem{lemma}[example]{Lemma}
\newtheorem{corollary}[example]{Corollary}

\newtheorem{remark}[example]{Remark}
\numberwithin{equation}{section}

\begin{document}

\changetext{}{+2.0cm}{-1.0cm}{-1.0cm}{}

\title{ \bf  Singularities of Connection Ricci Flow and Ricci Harmonic Flow}
\author{Pengshuai Shi}
\date{}

\maketitle

\begin{abstract}
In this paper, we study the singularities of two extended Ricci flow
systems --- connection Ricci flow and Ricci harmonic flow using
newly-defined curvature quantities. Specifically, we give the
definition of three types of singularities and their corresponding
singularity models, and then prove the convergence. In addition, for
Ricci harmonic flow, we use the monotonicity of functional
$\nu_\alpha$ to show the connection between finite-time singularity
and shrinking Ricci harmonic soliton. At last, we explore the
property of ancient solutions for Ricci harmonic flow.
\end{abstract}

\section{Introduction}

The Ricci flow theory which was founded by Richard Hamilton in 1980s has a big influence on
differential geometry. It is a very good tool to study the structure and characterization of
some manifolds and led to the solutions of Poincar${\rm \acute{e}}$ conjecture and geometrization
conjecture. When considering the solutions of the Ricci flow equation, it is inevitable to
make contact with singularities. That is, the flow stops as a result of the degeneration of
some geometric quantities. In this situation, in order to make the flow continue past the
singularities, we should adopt the method of geometric surgery introduced by Grigori Perelman.
This is the key point in proving the Poincar${\rm \acute{e}}$ conjecture

On the basis of Ricci flow, in recent years, some extended Ricci flow systems
began to be researched by people. The followings are two of them.

The first one is connection Ricci flow (CRF). Its equation is
\begin{eqnarray}
\frac{\partial}{\partial t}g\!\!\!&=&\!\!\! -2{\rm
Rc}+\frac{1}{2}\mathcal{H},
\nonumber \\
\frac{\partial}{\partial t}H\!\!\!&=&\!\!\!\Delta_{\rm LB}H,
\end{eqnarray}
where $H$ is a closed 3-form on manifold $M^n$, $\mathcal{H}_{ij}=g^{pq}g^{rs}
H_{ipr}H_{jqs}$, and $\Delta_{{\rm LB}}=-dd^\ast-d^\ast d$ represents the
Laplace-Beltrami operator.

The second one is Ricci harmonic flow. Its equation is
\begin{eqnarray}
\frac{\partial}{\partial t}g\!\!\!&=&\!\!\!-2{\rm
Rc}+2\alpha\nabla\phi\otimes\nabla\phi,
\nonumber \\
\frac{\partial}{\partial t}\phi\!\!\!&=&\!\!\!\tau_g\phi,
\end{eqnarray}
where $\phi(t):M^m\to N^n$ is a family of smooth maps between Riemannian manifolds,
$\tau_g\phi$ is the tension field of $\phi$ with respect to $g$, $\alpha$ is
nonnegative and time-dependent in general. But in this paper, we assume that $\alpha\ge0$
is a time-independent constant.

Connection Ricci flow is a special case of the renormalization group flow in physics,
and can also be seen as the Ricci flow in manifolds with non-trivial torsion. [St] and
[Li] studied the short-time existence, evolution equations of curvature, derivative
estimate, variational structure and compactness theorem from different angles.

Ricci harmonic flow arises from the combination of Ricci flow and harmonic map flow.
Main work is done by [Mue] and [List]. In addition to the results above, they also
gave the primary property of singularity, no breathers theorem and non-collapsing theorem,
and so on.

Considering the above two flows in mathematical sense, if we can prove their long-time
convergence under some circumstances, then the limit of connection Ricci flow may be a manifold
with constant curvature and torsion, and the limit of Ricci harmonic flow may be a manifold
with constant curvature and a harmonic map from it to another manifold. But analogous to
Ricci flow, in order to study the convergence of long-time solution, a problem that must
be solved is the singularity. Unfortunately, so far, there is not very deep research about
the property of singularities of the two flows. Even the natural classification of the
singularities has not been obtained. In this paper, we relate the specific properties of the
two flows to the work by Hamilton in Ricci flow. Through introducing a new curvature quantity
and using it to establish the singularity models, we successfully classify the singularities
of connection Ricci flow and Ricci harmonic flow into three relatively simple forms ---
ancient solutions, eternal solutions and immortal solutions. In addition, for Ricci harmonic
flow, we combine the finite-time singularity to soliton which can be seen as the self-similar
solution of the flow and also do some research about the ancient solutions. The work in this
paper can be helpful in understanding the long-time behavior of the two flows. In the end, we 
want to point out that our method here may also apply to the singularity problem of some other 
flows.

\subsection{Structure and main results}

Now we introduce the organization of the paper and state the main results.

Section 2 is a general review of connection Ricci flow (mainly in [St], [Li])
and Ricci harmonic flow (mainly in [Mue], [List]).

Section 3 is about the singularities of connection Ricci flow. In \S3.1, we give the
definition of maximum solution regarding the unboundedness of $|Rm|$. In \S3.2, we define
a new curvature quantity $P$ and use it to classify the singularities into three types.
Then we establish three singularity models and prove the convergence theorem.

\begin{theorem}
For any maximal solution to the connection Ricci flow which satisfies the
injectivity radius estimate and is of Type {\rm \uppercase\expandafter{\romannumeral 1}},
{\rm \uppercase\expandafter{\romannumeral 2}a}, {\rm \uppercase\expandafter{\romannumeral 2}b},
or {\rm \uppercase\expandafter{\romannumeral 3}}, there exists a sequence of dilations of
the solution which converges in the $C^\infty_{\rm loc}$ topology to a singularity model
of the corresponding type.
\end{theorem}

Section 4 is about the singularities of Ricci harmonic flow. In \S4.1, we also define three
types of singularities and three singularity models by a newly-defined curvature quantity
$Q$ and get the following conclusion.

\begin{theorem}
For any maximal solution to Ricci harmonic flow which satisfies the injectivity
radius estimate and is of Type {\rm \uppercase\expandafter{\romannumeral 1}},
{\rm \uppercase\expandafter{\romannumeral 2}a}, {\rm \uppercase\expandafter{\romannumeral 2}b},
or {\rm \uppercase\expandafter{\romannumeral 3}}, there exists a sequence of dilations of
the solution which converges in the $C^\infty_{\rm loc}$ topology to a singularity model
of the corresponding type.
\end{theorem}

In \S4.2, we prove a theorem concerning the finite-time singularity and shrinking soliton on
closed manifolds.

\begin{theorem}
Let
$(g(t),\phi(t))_{t\in[0,T)}$, be a maximal solution to the Ricci
harmonic flow (1.2) on a closed manifold $M^m$ with singular time
$T<\infty$. Let $t_k\to T$ be a sequence of times such that
$Q_k=Q(p_k,t_k)\to\infty$. If the rescaled sequence
$(M^m,Q_kg(t_k+Q_k^{-1}t),\phi(t_k+Q_k^{-1}t),p_k)$ converges in the
$C^\infty$ sense to a closed ancient solution
$(M^m_\infty,g_\infty(t),\phi_\infty(t),p_{\infty})$ to the Ricci
harmonic flow, then $(g_\infty(t),\phi_\infty(t))$ must be a
shrinking Ricci harmonic soliton.
\end{theorem}

In the third part, we study the ancient solution of Ricci harmonic flow by dealing with
the compact case and noncompact case respectively.

\begin{theorem}
Let $(g(t),\phi(t))$ be an ancient solution to Ricci harmonic flow (1.2) on manifold
$M^m$, then for any $t$ such that the solution exists, we have $S=R-\alpha|\nabla\phi|^2\ge0$.
That is, the scalar curvature is always nonnegative.
\end{theorem}

\noindent{\bf Acknowledgements: }This paper is based on the author's master degree thesis in Zhejiang
University. The author would like to thank his adviser, Professor Weimin Sheng, for a number of
helpful talks and encouragements. The author also feels very grateful to the teachers and students
in Zhejiang University for what he learnt from them.

\section{Preliminaries}

In this section, we collect and derive some results about connection Ricci flow and Ricci
harmonic flow. Except for Theorem 2.9, all the results have been well known before.

\subsection{Connection Ricci flow}

Connection Ricci flow is a generalization of Ricci flow to connections with torsion and
can be seen as a special case of renormalization group flow which has a physical background.
Here we give some main results that will be used in this paper and we suggest the readers to
refer [St] and [Li] for more details.

In a Riemannian manifold $(M,g)$, a general connection is defined as
\[
\tau (X,Y)=\nabla_X Y-\nabla_Y X-[X,Y].
\]
In particular, when $\tau=0$, it is the usual Levi-Civita connection. In [St], Streets
let the torsion be \emph{geometric} (that is, $g_{kl}\tau_{ij}^k$ is a 3-form and $d\tau=0$)
and study this kind of torsions. For a geometric torsion $\tau$, after computation, its
curvature tensor is as following
\begin{eqnarray}
R^h_{jkl}\!\!\!&=&\!\!\!\widetilde{R}^h_{jkl}
+\frac{1}{2}\bigg(\frac{\partial \tau^h_{lj}}{\partial
x^k}-\frac{\partial \tau^h_{kj}}{\partial x^l}\bigg)
+\frac{1}{2}(\widetilde{\Gamma}^h_{kp}\tau^p_{lj}+\tau^h_{kp}\widetilde{\Gamma}^p_{lj})
\nonumber \\
\!\!\!&&\!\!\!+\frac{1}{4}\tau^h_{kp}\tau^p_{lj}
-\frac{1}{2}(\widetilde{\Gamma}^h_{lp}\tau^p_{kj}+\tau^h_{lp}\widetilde{\Gamma}^p_{kj})
-\frac{1}{4}\tau^h_{lp}\tau^p_{kj},
\end{eqnarray}
where quantities with ¡°$\thicksim$¡± are in Levi-Civita connection. From this, the Ricci
curvature tensor can be got. It is not symmetric, but has the following symmetric part
and skew-symmetric part
\begin{eqnarray}
{\rm Rc}^\otimes\!\!\!&=&\!\!\!{\rm Rc}-\frac{1}{4}\mathcal{H},
\nonumber \\
{\rm Rc}^\wedge\!\!\!&=&\!\!\!-\frac{1}{2}d^\ast \tau.
\end{eqnarray}

With these preparations, we can consider the connection Ricci flow
\begin{eqnarray*}
\frac{\partial}{\partial t}g\!\!\!&=&\!\!\!-2{\rm Rc}^\otimes,\\
\frac{\partial}{\partial t}\tau\!\!\!&=&\!\!\!2d{\rm Rc}^\wedge.
\end{eqnarray*}
In another form, it is
\begin{eqnarray*}
\frac{\partial}{\partial t}g\!\!\!&=&\!\!\!-2{\rm Rc}+\frac{1}{2}\mathcal{H},\\
\frac{\partial}{\partial t}\tau\!\!\!&=&\!\!\!\Delta _{\rm LB}\tau.
\end{eqnarray*}
Throughout this paper, we use notation $H$ instead of $\tau$. Then the connection Ricci
flow equation is exactly (1.1).

[St] and [Li] studied connection Ricci flow equation using different methods. The former
viewed the equation as the Ricci flow in general manifolds with torsion while the
latter considered the equation in torsion-free manifolds and regarded $H$ as a separated
3-form. Both methods serve to get some properties about the flow including short-time
existence of solution, evolution equations, compactness theorem and some functionals.
In this paper, we adopt Li's method.

The evolution equations under connection Ricci flow are

\begin{proposition}
Under (1.1),
\begin{eqnarray}
\frac{\partial}{\partial t}R_{ijkl}\!\!\!&=&\!\!\!\Delta
R_{ijkl}+2(B_{ijkl}-B_{ijlk}-B_{iljk}+B_{ikjl})
\nonumber \\
\!\!\!& &\!\!\!
-(R_{pjkl}R_{pi}+R_{ipkl}R_{pj}+R_{ijpl}R_{pk}+R_{ijkp}R_{pl})
\nonumber \\
\!\!\!& &\!\!\!
-\frac{1}{4}[\nabla_i\nabla_k\mathcal{H}_{jl}-\nabla_i\nabla_l\mathcal{H}
_{jk}-\nabla_j\nabla_k\mathcal{H}_{il}+\nabla_j\nabla_l\mathcal{H}_{ik}]
\nonumber \\
\!\!\!& &\!\!\!
+\frac{1}{4}[R_{ijkp}\mathcal{H}_{pl}+R_{ijpl}\mathcal{H}_{pk}], \\
\frac{\partial}{\partial t}R_{ij}\!\!\!&=&\!\!\!\Delta
R_{ij}+2R_{piqj}R_{pq}-2R_{pi}R_{pj}-\frac{1}{4}[
R_{piqj}\mathcal{H}_{pq}-R_{pi}\mathcal{H}_{pj}]
\nonumber \\
\!\!\!& &\!\!\!
-\frac{1}{4}\big[\nabla_i\nabla_j|H|^2-\nabla_i\nabla_p\mathcal{H}
_{pj}-\nabla_p\nabla_j\mathcal{H}_{ip}+\Delta\mathcal{H}_{ij}\big], \\
\frac{\partial}{\partial
t}R\!\!\!&=&\!\!\!\Delta R+2|{\rm Rc}|^2-\frac{1}{2} \langle {\rm
Rc},\mathcal{H}\rangle-\frac{1}{2}\Delta|H|^2
+\frac{1}{2}g^{ik}g^{jl}\nabla_i\nabla_j\mathcal{H}_{kl}, \\
\frac{\partial}{\partial
t}\mathcal{H}_{ij}\!\!\!&=&\!\!\!2\langle\Delta_{\rm
LB}H_{ikl},H_{jkl}\rangle+4\langle
R_{ln}-\frac{1}{4}\mathcal{H}_{ln},H_{ikl}H_{jkn}\rangle, \\
\frac{\partial}{\partial
t}|H|^2\!\!\!&=&\!\!\!2\langle\Delta_{\rm LB}H,H\rangle+6\langle{\rm
Rc},\mathcal{H}\rangle-\frac{3}{2}|\mathcal{H}|^2,
\end{eqnarray}
where $B_{ijkl}=R_{piqj}R_{pkql}$.
\end{proposition}

By (2.7) and $\Delta_{\rm LB}H=\Delta H+Rm\ast H$, we get

\begin{theorem}
Let $(M^n,g(x,t),H(x,t))_{t\in[0,T]}$ be a complete solution to connection Ricci
flow, and $K_1,K_2$ are arbitrary given nonnegative constants. If
\[
\sup_{M^n\times[0,T]}|Rm(x,t)|_{g(x,t)}\le K_1,\quad
\sup_{M^n}|H(x,0)|^2_{g(x,0)}\le K_2
\]
for all $x\in M^n$ and $t\in[0,T]$, then there exists a constant $C_n$ depending only
on $n$ such that
\[
\sup_{M^n\times[0,t]}|H(x,t)|^2_{g(x,t)}\le K_2{\rm e}^{C_nK_1T}
\]
for all $x\in M^n$ and $t\in[0,T]$.
\end{theorem}

The derivative estimates

\begin{theorem}
Let $(M^n,g(x,t),H(x,t))$ be a complete solution to connection Ricci
flow, and $K$ is an arbitrary given positive constant. Then for each $\beta>0$ and each
integer $l\ge 1$ there exists a constant $C_l$ depending only on $l,n,\max{\{\beta,1\}}$
and $K$ such that if
\[
|Rm(x,t)|_{g(x,t)}\le K,\quad |H(x,0)|_{g(x,0)}^2\le K
\]
for all $x\in M^n$ and $t\in[0,\frac{\beta}{K}]$, then
\[
|\nabla^{l-1}Rm(x,t)|_{g(x,t)}+|\nabla^l H(x,t)|_{g(x,t)}\le
\frac{C_l}{t^{l/2}}
\]
for all $x\in M^n$ and $t\in(0,\frac{\beta}{K}]$.
\end{theorem}

In addition, [Zheng] gives the locally derivative estimate.

Then we show the compactness theorem.

\begin{theorem}
Let $(M_k^n,g_k(t),H_k(t),p_k)$ be a sequence of
complete pointed solutions to connection Ricci flow for $t\in(A,\Omega)\ni0$, such
that
\begin{itemize}
\item[{\rm (\romannumeral 1)}] there is a constant $C_0<\infty$ independent of $k$
such that
\[
\sup_{(x,t)\in M_k^n\times(A,\Omega)}|Rm_k(x,t)|_{g_k(x,t)}\leq C_0,\quad\sup_{x\in
M_k^n}|H_k(x,0)|_{g_k(x,0)}\leq C_{0},
\]
\item[{\rm (\romannumeral 2)}] there is a constant $\iota _0>0$ satisfies
\[
{\rm inj}_{g_k(0)}(p_k)\ge\iota _0>0.
\]
\end{itemize}
Then there exists a subsequence $\{j_k\}$ such that $(M_{j_{k}}^n,g_{j_{k}}(t),
H_{j_{k}}(t),p_{j_{k}})$ converges to a complete pointed solution $(M_\infty^n,
g_\infty(t),H_\infty(t),p_\infty)$, $t\in(A,\Omega)$ to connection
Ricci flow as $k\to\infty$.
\end{theorem}

\subsection{Ricci harmonic flow}

Ricci harmonic flow is a combination of Ricci flow and harmonic map flow, where the latter
is an important tool to study harmonic map. [List] and [Mue] did detailed research about
this topic. Here we mainly introduce the evolution equations and some results about solitons
and maximum solutions. After this, we prove the convergence theorem.

\begin{proposition}
Under (1.2),
\begin{eqnarray}
\frac{\partial}{\partial t}R_{ijkl}\!\!\!&=&\!\!\!\Delta
R_{ijkl}+2(B_{ijkl}-B_{ijlk}-B_{iljk}+B_{ikjl})
\nonumber \\
\!\!\!& &\!\!\!
-(R_{pjkl}R_{pi}+R_{ipkl}R_{pj}+R_{ijpl}R_{pk}+R_{ijkp}R_{pl})
\nonumber \\
\!\!\!& &\!\!\!
+2\alpha(\nabla_i\nabla_k\phi\nabla_j\nabla_l\phi-\nabla_i\nabla_l\phi
\nabla_j\nabla_k\phi
\nonumber \\
\!\!\!& &\!\!\!
-\langle^NRm(\nabla_i\phi,\nabla_j\phi)\nabla_k\phi,
\nabla_l\phi\rangle), \\
\frac{\partial}{\partial t}R_{ij}\!\!\!&=&\!\!\!\Delta
R_{ij}+2R_{piqj}R_{pq}-2R_{pi}R_{pj}-2\alpha R_{piqj}\nabla_p\phi\nabla_q\phi
\nonumber \\
\!\!\!& &\!\!\!
-2\alpha\nabla_p\nabla_i\phi\nabla_p\nabla_j\phi
+2\alpha\tau_g\phi\nabla_i\phi\nabla_j\phi
\nonumber \\
\!\!\!& &\!\!\!
+2\alpha\langle^N
Rm(\nabla_i\phi,\nabla_p\phi)\nabla_p\phi,\nabla_j\phi\rangle, \\
\frac{\partial}{\partial
t}R\!\!\!&=&\!\!\!\Delta R+2|{\rm Rc}|^2-4\alpha\langle{\rm Rc},
\nabla\phi\otimes\nabla\phi\rangle-2\alpha|\nabla^2\phi|^2+2\alpha
|\tau_g\phi|^2
\nonumber \\
\!\!\!& &\!\!\!+2\alpha\langle^NRm(\nabla_i\phi,\nabla_j\phi)\nabla_j\phi,
\nabla_i\phi\rangle, \\
\frac{\partial}{\partial
t}(\nabla_i\phi\nabla_j\phi)\!\!\!&=&\!\!\!\Delta(\nabla_i\phi\nabla_j\phi)
-R_{pi}\nabla_p\phi\nabla_j\phi-R_{pj}\nabla_p\phi\nabla_i\phi
-2\nabla_p\nabla_i\phi\nabla_p\nabla_j\phi
\nonumber \\
\!\!\!& &\!\!\!+2\langle^NRm(\nabla_i\phi,\nabla_p\phi)\nabla_p\phi,
\nabla_j\phi\rangle, \\
\frac{\partial}{\partial
t}|\nabla\phi|^2\!\!\!&=&\!\!\!\Delta|\nabla\phi|^2-2|\nabla^2\phi|^2
-2\alpha|\nabla\phi\otimes\nabla\phi|^2
\nonumber \\
\!\!\!& &\!\!\!
+2\langle^NRm(\nabla_i\phi,\nabla_j\phi)
\nabla_j\phi,\nabla_i\phi\rangle,
\end{eqnarray}
where $B_{ijkl}=R_{piqj}R_{pkql}$, and $^NRm$ represents the curvature tensor of $N^n$.
\end{proposition}

\emph{If we set $S_{ij}=R_{ij}-\alpha\nabla_i\phi\nabla_j\phi,\;S=R-\alpha|\nabla\phi|^2$, then
\begin{eqnarray}
\frac{\partial}{\partial t}S_{ij}\!\!\!&=&\!\!\!\Delta S_{ij}
+2R_{piqj}S_{pq}-R_{pi}S_{pj}-R_{pj}S_{pi}+2\alpha\tau_g\phi\nabla_i\nabla_j\phi, \\
\frac{\partial}{\partial t}S\!\!\!&=&\!\!\!\Delta S+2|S_{ij}|^2+2\alpha
|\tau_g\phi|^2.
\end{eqnarray}}

\begin{definition}
A solution $(g(t),\phi(t))_{t\in[0,T)}$ of (1.2) is called a soliton if there exists
a one-parameter family of diffeomorphisms $\psi_t:M^m\to M^m$ with $\psi_0={\rm
id}_{M^m}$ and a scaling function $c:[0,T)\to\mathbb{R}_+$ such that
\[
\left\{
\begin{array}{r l}
g(t) & \!\!\!\!=c(t)\psi^*_tg(0), \\
\phi(t) & \!\!\!\!=\psi^*_t\phi(0).
\end{array}
\right.
\]
The case $\frac{\partial}{\partial t}c=\dot{c}<0,\;\dot{c}=0$ and $\dot{c}>0$
correspond to shrinking, steady and expanding solitons, respectively. If the diffeomorphisms
$\psi_t$ are generated by a vector field $X(t)$ that is a gradient of some function $f(t)$
on $M^m$, then the soliton is called gradient soliton and $f$ is called the potential of the
soliton.
\end{definition}

Given a closed manifold $(M^m,g)$ and a map $\phi$ from $M^m$ to a Riemannian manifold $N^n$,
define
\[
\lambda_\alpha(g,\phi)=\inf\bigg\{\int_{M^m}(R+|\nabla
f|^2-\alpha|\nabla\phi|^2)e^{-f} dV\Big|\int_{M^m}e^{-f}dV=1\bigg\}.
\]
for $\tau>0$, again define
\[
\mu_\alpha(g,\phi,\tau)=\inf\bigg\{\int_{M^m}[\tau(R+|\nabla
f|^2-\alpha|\nabla\phi|^2)
+f-m](4\pi\tau)^{-m/2}e^{-f}dV\Big|\int_{M^m}(4\pi\tau)^{-m/2}e^{-f}dV=1\bigg\},
\]
where $R$ represents the scalar curvature with respect to $g$, the infimum is attained
throughout all functions $f\in C^\infty(M^m)$.

It is easy to know that $\lambda_\alpha$ is the first eigenvalue of the operator $-4\Delta+R-\alpha|\nabla\phi|^2$. Similar to Ricci flow, for each $\tau>0$, there exists
a smooth minimizer of $\mu_\alpha(g,\phi,\tau)$. The functional $\nu_\alpha$ is defined by
\[
\nu_\alpha(g,\phi)=\inf_{\tau>0}\mu_\alpha(g,\phi,\tau).
\]
By \S4.2, we know $\nu_\alpha(g,\phi)$ may be $-\infty$.

For the functional $\nu_\alpha$, we have the following significant monotonicity proposition
under Ricci harmonic flow.

\begin{proposition}
Let $(g(t),\phi(t))$ be a solution to (1.2) on a closed manifold $M^m$ with $\alpha$ a
constant. For a constant $T>0$, set $\tau(t)=T-t$, then $\mu_\alpha(g(t),\phi(t),\tau(t))$
is nondecreasing whenever it exists. Moreover, the monotonicity is strict unless
$(g(t),\phi(t))$ is a shrinking soliton.
\end{proposition}

As for the long-time existence of Ricci harmonic flow, there is the following theorem.

\begin{theorem}
Let $(g(t),\phi(t))_{t\in[0,T)}$ be a solution to (1.2). Suppose $T<\infty$ is chosen such
that the solution is maximal, i.e. the solution cannot be extended beyond $T$ in a smooth
way. Then the curvature of $(M^m,g(t))$ has to become unbounded for $t\nearrow T$ in the
sense that
\[
\limsup_{t\nearrow T}(\max_{x\in M^m}|Rm(x,t)|^2)=\infty.
\]
\end{theorem}

In order to deal with the singularity model, we also need the compactness theorem
for Ricci harmonic flow. The following theorem is what we will prove in this section.

\begin{theorem}
Let $(M_k^m,g_k(t),\phi_k(t),p_k)$, $t\in(A,\Omega)
\ni0$, be a sequence of complete pointed solutions to Ricci harmonic flow such that
\begin{itemize}
\item[{\rm (\romannumeral 1)}] there is a constant $C_0<\infty$ independent of $k$
such that
\[
\sup_{(x,t)\in M_k^m\times(A,\Omega)}|Rm_k(x,t)|_{g_k(x,t)}\leq C_0,
\]
\item[{\rm (\romannumeral 2)}] there is a constant $\iota _0>0$ satisfies
\[
{\rm inj}_{g_k(0)}(p_k)\ge\iota _0>0.
\]
\end{itemize}
Then there exists a subsequence $\{j_k\}$ such that $(M_{j_k}^m,g_{j_k}(t),\phi_{j_k}(t),
p_{j_k})$ converges to a complete pointed solution to Ricci harmonic flow $(M_\infty^m,
g_\infty(t),\phi_\infty(t),p_\infty)$, $t\in(\alpha,\omega)$, as $k\to\infty$.
\end{theorem}

The proof is a standard procedure given by Richard Hamilton which has already solves
the compactness theorem in Ricci flow and connection Ricci flow. Here we only need to
prove the following lemma which plays a vital role in the proof of the theorem.

\begin{lemma}
Let $(M^m,g)$ be a Riemannian manifold, $K$ a compact
subset of $M^m$, and $(g_k(t),\phi_k(t))$ a collection of solutions to Ricci harmonic
flow defined on neighborhoods of $K\times[\gamma,\delta]$ with $[\gamma,\delta]$ containing
0. Suppose that for each $r$,
\begin{itemize}
\item[{\rm (a)}] $C_0^{-1}g\le g_k(0)\le C_0g$, \;on $K$, \;for all $k$,
\item[{\rm (b)}] $|\nabla^rg_k(0)|+|\nabla^r\phi_k(0)|\le C_r$, \;on $K$, \;for all $k$,
$r=1,2,\cdots$,
\item[{\rm (c)}] $|\nabla_k\phi_k|_k\le C'$, $|\nabla_k^rRm_k|_k+|\nabla_k^{r+2}\phi_k|_k\le
C_r'$, \;on $K\times[\gamma,\delta]$, \;for all $k$, $r=0,1,2,\cdots$,
\end{itemize}
for some positive constants $C'$, $C_r$, $C_r'$ independent of $k$, where
$Rm_k$ are the curvature tensor of the metrics $g_k(t)$, $\nabla_k$ denote covariant
derivative with respect to $g_k(t)$, $|\cdot|_k$ are the length of a tensor with respect
to $g_k(t)$, and $|\cdot|$ is the length with respect to $g$. Then the metrics $g_k(t)$
satisfy
\[
\tilde{C}_0^{-1}g\le g_k(t)\le\tilde{C}_0g,\;\mbox{on }K\times[\gamma,\delta]
\]
and
\[
\bigg|\frac{\partial^s}{\partial t^s}\nabla^rg_k\bigg|+\bigg|\frac{\partial^s}{\partial t^s}\nabla^r\phi_k\bigg|\le\tilde{C}_{r,s},\;\mbox{on }K\times[\gamma,\delta],\;r,s=1,2,\cdots,
\]
for all $k$, where $\tilde{C}_{r,s}$ are positive constants independent of $k$.
\end{lemma}

\begin{proof}
First, the equation
\[
\frac{\partial}{\partial t}g_k=-2{\rm Rc}_k+2\alpha\nabla\phi_k\otimes\nabla\phi_k
\]
and the assumption (c) give that
\begin{equation}
\tilde{C}_0^{-1}g\le g_k(t)\le\tilde{C}_0g,
\end{equation}
on $K\times[\alpha,\beta]$ for some positive constant $\tilde{C}_0$ independent of $k$.

For the second part of the conclusion. When $s=0$, we divide the proof into two parts.
That is, we will prove the following
\[
|\nabla^rg_k|\le\tilde{C}_{r,0}',\quad|\nabla^r\phi_k|\le\tilde{C}_{r,0}''.
\]
Here we show the case of $r=1,2$, and the higher covariant derivative cases can be
derived by the same argument.

Taking the difference of the connection $\Gamma_k$ of $g_k$ and the connection $\Gamma$
of $g$ with $\Gamma$ being fixed in time, we get
\begin{eqnarray*}
&&\!\!\!\frac{\partial}{\partial t}((\Gamma_k)_{ab}^c-\Gamma_{ab}^c) \\
&=&\!\!\!\frac{\partial}{\partial t}\bigg(\frac{1}{2}(g_k)^{cd}\bigg[(\nabla_k)_a(g_k)_{bd}
+(\nabla_k)_b(g_k)_{ad}-(\nabla_k)_d(g_k)_{ab}\bigg]\bigg) \\
&=&\!\!\!\frac{1}{2}(g_k)^{cd}[(\nabla_k)_a(-2({\rm Rc}_k)_{bd}+2\alpha(\nabla_k)_b
\phi_k\cdot(\nabla_k)_d\phi_k) \\
&&\qquad\;\;+(\nabla_k)_b(-2({\rm Rc}_k)_{ad}+2\alpha(\nabla_k)_a\phi_k\cdot(\nabla_k)_d
\phi_k) \\
&&\qquad\;\;-(\nabla_k)_d(-2({\rm Rc}_k)_{ab}+2\alpha(\nabla_k)_a\phi_k\cdot(\nabla_k)_b
\phi_k)],
\end{eqnarray*}
and by assumption (c) and (2.15),
\[
\bigg|\frac{\partial}{\partial t}(\Gamma_k-\Gamma)\bigg|\le C,\quad\mbox{for all }k.
\]
For time $t=0$, at a normal coordinate of the metric $g$ at a fixed point, note that
\begin{equation}
(\Gamma_k)_{ab}^c-\Gamma_{ab}^c=\frac{1}{2}(g_k)^{cd}(\nabla_a(g_k)_{bd}+\nabla_b
(g_k)_{ad}-\nabla_d(g_k)_{ab}),
\end{equation}
then by assumption (b) and (2.15)
\[
|\Gamma_k(0)-\Gamma|\le C,\quad\mbox{for all }k.
\]
Integrating over time we deduce that
\begin{equation}
|\Gamma_k-\Gamma|\le C,\;\mbox{on }K\times[\gamma,\delta],\quad\mbox{for all }k.
\end{equation}
Again using assumption (c), (2.15) and (2.17), we have
\begin{eqnarray*}
\bigg|\frac{\partial}{\partial t}(\nabla g_k)\bigg|\!\!\!&=&\!\!\!|-2\nabla
({\rm Rc}_k-\alpha\nabla_k\phi_k\otimes\nabla_k\phi_k)| \\
&=&\!\!\!|-2\nabla_k({\rm Rc}_k-\alpha\nabla_k\phi_k\otimes\nabla_k\phi_k) \\
&&+(\Gamma_k-\Gamma)\ast({\rm Rc}_k-\alpha\nabla_k\phi_k\otimes\nabla_k\phi_k)| \\
&\le&\!\!\!C,\quad\mbox{for all }k.
\end{eqnarray*}
Thus by combining with assumption (b) we get bounds
\begin{equation}
|\nabla g_k|\le\tilde{C}_{1,0}',\quad\mbox{on }K\times[\gamma,\delta],
\end{equation}
where $\tilde{C}_{1,0}'$ is a positive constant independent of $k$.

Similarly, assumption (c), (2.15) and (2.17) also indicate
\begin{eqnarray*}
\bigg|\frac{\partial}{\partial t}(\nabla\phi_k)\bigg|\!\!\!&=&\!\!\!|\nabla(\tau
_{g_{k}}\phi_k)| \\
&=&\!\!\!|\nabla_k(\tau_{g_{k}}\phi_k)+(\Gamma_k-\Gamma)\ast(\tau_{g_{k}}\phi_k)| \\
&\le&\!\!\!C,\quad\mbox{for all }k.
\end{eqnarray*}
This, combined with assumption (b), gives the bounds
\[
|\nabla\phi_k|\le\tilde{C}_{1,0}'',\quad\mbox{on }K\times[\gamma,\delta],
\]
where $\tilde{C}_{1,0}''$ is a positive constant independent of $k$. The case of
$r=1$ is finished.

Next we begin to deal with the case of $r=2$. In order to bound $\nabla^2g_k$, again
regarding $\nabla$ as fixed in time, we know
\[
\frac{\partial}{\partial t}(\nabla^2g_k)=-2\nabla^2({\rm Rc}_k-\alpha\nabla_k\phi_k
\otimes\nabla_k\phi_k).
\]
Write
\begin{eqnarray*}
&&\!\!\!\nabla^2({\rm Rc}_k-\alpha\nabla_k\phi_k
\otimes\nabla_k\phi_k) \\
&=&\!\!\!\big[(\nabla-\nabla_k)\nabla+\nabla_k(\nabla-\nabla_k)+\nabla_k^2\big]({\rm Rc}_k-
\alpha\nabla_k\phi_k\otimes\nabla_k\phi_k) \\
&=&\!\!\!(\Gamma-\Gamma_k)\ast(\nabla({\rm Rc}_k-\alpha\nabla_k\phi_k\otimes\nabla_k\phi_k)) \\
&&\!\!\!+\nabla_k((\Gamma-\Gamma_k)\ast({\rm Rc}_k-\alpha\nabla_k\phi_k\otimes\nabla_k\phi_k)) \\
&&\!\!\!+\nabla_k^2({\rm Rc}_k-\alpha\nabla_k\phi_k\otimes\nabla_k\phi_k) \\
&=&\!\!\!(\Gamma-\Gamma_k)\ast((\Gamma-\Gamma_k)\ast({\rm Rc}_k-\alpha\nabla_k\phi_k
\otimes\nabla_k\phi_k)+\nabla_k({\rm Rc}_k-\alpha\nabla_k\phi_k\otimes\nabla_k\phi_k)) \\
&&\!\!\!+\nabla_k(g_k^{-1}\ast\nabla g_k\ast({\rm Rc}_k-\alpha\nabla_k\phi_k\otimes\nabla_k
\phi_k)) \\
&&\!\!\!+\nabla_k^2({\rm Rc}_k-\alpha\nabla_k\phi_k\otimes\nabla_k\phi_k),
\end{eqnarray*}
where we have used (2.16). Then using assumption (c), (2.15), (2.17) and (2.18), we have
\begin{eqnarray*}
\bigg|\frac{\partial}{\partial t}\nabla^2g_k\bigg|\!\!\!&\le&\!\!\!C+C\cdot|\nabla_k\nabla
g_k| \\
&=&\!\!\!C+C\cdot|\nabla^2g_k+(\Gamma_k-\Gamma)\ast\nabla g_k| \\
&\le&\!\!\!C+C|\nabla^2g_k|.
\end{eqnarray*}
Thus by combining with assumption (b) we get bounds
\[
|\nabla^2g_k|\le\tilde{C}_{2,0}',\quad\mbox{on }K\times[\gamma,\delta],
\]
where $\tilde{C}_{2,0}'$ is a positive constant independent of $k$.

Similarly, assumption (c), (2.15) and (2.17) also indicate
\begin{eqnarray*}
\bigg|\frac{\partial}{\partial t}(\nabla^2\phi_k)\bigg|\!\!\!&=&\!\!\!|\nabla^2(\tau
_{g_{k}}\phi_k)| \\
&\le&\!\!\!C+C|\nabla^2g_k| \\
&\le&\!\!\!C,\quad\mbox{for all }k.
\end{eqnarray*}
This, combined with assumption (b), gives the bounds
\[
|\nabla^2\phi_k|\le\tilde{C}_{2,0}'',\quad\mbox{on }K\times[\gamma,\delta],
\]
where $\tilde{C}_{2,0}''$ is a positive constant independent of $k$. The case of
$r=2$ is also proved.

In the end, we discuss the case $s\ge1$ in the second part of the conclusion. Since
\begin{eqnarray*}
\frac{\partial^s}{\partial t^s}\nabla^rg_k\!\!\!&=&\!\!\!\nabla^r\frac{\partial^{s-1}}
{\partial t^{s-1}}(-2{\rm Rc}_k+2\alpha\nabla_k\phi_k\otimes\nabla_k\phi_k), \\
\frac{\partial^s}{\partial t^s}\nabla^r\phi_k\!\!\!&=&\!\!\!\nabla^r\frac{\partial^{s-1}}
{\partial t^{s-1}}(\tau_{g_k}\phi_k),
\end{eqnarray*}
using the evolution equations for curvature, by induction, we can see that the above
two quantities are bounded by a sum of terms which are products of $|\nabla^{r_1}\nabla_k
^{s_1}Rm_k|$, $|\nabla^{r_2}\nabla_k^{s_2}{\rm Rc_k}|$, $|\nabla^{r_3}\nabla_k^{s_3}R_k|$
and $|\nabla^{r_4}\nabla_k^{s_4}\phi_k|$. Hence we get
\[
\bigg|\frac{\partial^s}{\partial t^s}\nabla^rg_k\bigg|\le\tilde{C}_{r,s}',\quad
\bigg|\frac{\partial^s}{\partial t^s}\nabla^r\phi_k\bigg|\le\tilde{C}_{r,s}'',\quad
\mbox{on }K\times[\gamma,\delta].
\]
\end{proof}

\section{Singularities of connection Ricci flow}

\subsection{Maximal solution of connection Ricci flow}

Before studying the singularities of connection Ricci flow, first
giving the definition of maximal solution.

\begin{definition}
Suppose
$(g(t),H(t))_{t\in[0,T)}$ is a solution to the connection Ricci flow
(1.1) on $M^n$, where either $M^n$ is compact or at each time $t$
the metric $g(\cdot,t)$ is complete and has bounded curvature. We
say that $(g(t),H(t))$ is a maximal solution if either $T=\infty$ or
$T<\infty$ and
\[
\lim_{t\to T}\sup_{x\in M^n}|Rm(x,t)|=\infty.
\]
\end{definition}

\begin{remark}
In fact, on $M^n$, both the
curvature and the 3-form evolve under connection Ricci flow. Hence
the flow will end if any one of them goes to infinity. However,
according to Theorem 2.2, if $|H|^2$ goes to infinity, then it can
be indicated that $|Rm|$ must also go to infinity. So the definition
of maximal solution can still be given by the unboundedness of
$|Rm|$.
\end{remark}

\subsection{Three types of singularities and singularity models of connection Ricci flow}

By $\S$3.1, if connection Ricci flow has finite-time singularities, then
\[
\lim_{t\to T}\sup_{x\in M^n}|Rm(x,t)|=\infty,
\]
where $T$ is the maximum time.

However, unlike Ricci flow, for connection Ricci flow, we cannot know how $|Rm|$
goes to infinity near maximum time. To deal with this problem, we define a new
curvature quantity.

\noindent\textbf{Definition\;3.3.} \emph{For connection Ricci flow,
we call $P=|Rm|+|\nabla H|+|H|^2$ the \textbf{absolute compound curvature (AC curvature)}.}

Under this definition, if connection Ricci flow has finite-time singularities, then
\[
\lim_{t\to T}\sup_{x\in M^n}P(x,t)=\infty.
\]

\begin{remark}
By (2.1), the three terms in the AC curvature defined above are exactly terms in
the curvature on manifold which has $H$ as its torsion. So in the view of geometry,
the AC curvature stands for the real curvature bound here.
\end{remark}

\begin{proposition}
If $0\le t<T<\infty$ is the maximal interval
of existence of the solution $(M^n,g(t),H(t))$ to connection Ricci flow, there exists
a constant $c_0>0$ depending only on $n$ such that
\begin{equation}
\sup_{x\in M^n}P(x,t)\ge\frac{c_0}{T-t}.
\end{equation}
\end{proposition}

\begin{proof}
By Proposition 2.1,
\begin{eqnarray*}
\frac{\partial}{\partial t}Rm\!\!\!&=&\!\!\!\Delta Rm+Rm\ast
Rm+Rm\ast
H\ast H+H\ast\nabla^2H+\nabla H\ast\nabla H, \\
\frac{\partial}{\partial t}\nabla H\!\!\!&=&\!\!\!\Delta(\nabla
H)+Rm\ast\nabla H+\nabla Rm\ast H+H\ast H\ast\nabla H, \\
\frac{\partial}{\partial
t}|H|^2\!\!\!&\le&\!\!\!\Delta|H|^2-2|\nabla H|^2+C|Rm||H|^2+C|H|^4,
\end{eqnarray*}
so
\begin{eqnarray*}
\frac{\partial}{\partial
t}|Rm|^2\!\!\!&\le&\!\!\!\Delta|Rm|^2-2|\nabla
Rm|^2+C|Rm|^2|H|^2+C|Rm||H||\nabla^2H| \\
&&\!\!\!+C|Rm||\nabla H|^2+C|Rm|^3, \\
\frac{\partial}{\partial t}|\nabla
H|^2\!\!\!&\le&\!\!\!\Delta|\nabla H|^2-2|\nabla^2H|^2+C|Rm||\nabla
H|^2+C|\nabla Rm||H||\nabla H| \\
&&\!\!\!+C|H|^2|\nabla H|^2, \\
\frac{\partial}{\partial
t}|H|^4\!\!\!&\le&\!\!\!\Delta|H|^4+C|Rm||H|^4+C|H|^6.
\end{eqnarray*}
Then
\begin{eqnarray*}
&&\!\!\!\frac{\partial}{\partial t}(|Rm|^2+|\nabla H|^2+|H|^4) \\
&\le&\!\!\!\Delta(|Rm|^2+|\nabla H|^2+|H|^4)-2|\nabla
Rm|^2-2|\nabla^2H|^2
\\
&&\!\!\!+C|\nabla Rm||H||\nabla H|+C|Rm||H||\nabla^2H|+C|Rm||\nabla
H|^2 \\
&&\!\!\!+C|Rm|^2|H|^2+C|Rm||H|^4+C|H|^2|\nabla H|^2+C|Rm|^3+C|H|^6.
\end{eqnarray*}
Using Cauchy inequality,
\begin{eqnarray*}
|\nabla Rm|^2+C_1|H|^2|\nabla H|^2\!\!\!&\ge&\!\!\!C|\nabla Rm||H||\nabla H|, \\
|\nabla^2H|^2+C_2|Rm|^2|H|^2\!\!\!&\ge&\!\!\!C|Rm||H||\nabla^2H|.
\end{eqnarray*}
This gives
\begin{eqnarray*}
&&\!\!\!\frac{\partial}{\partial t}(|Rm|^2+|\nabla H|^2+|H|^4) \\
&\le&\!\!\!\Delta(|Rm|^2+|\nabla H|^2+|H|^4)+C|Rm||\nabla H|^2+C|Rm|^2|H|^2 \\
&&\!\!\!+C|Rm||H|^4+C|H|^2|\nabla H|^2+C|Rm|^3+C|H|^6 \\
&\le&\!\!\!\Delta(|Rm|^2+|\nabla H|^2+|H|^4)+C(|Rm|+|\nabla H|+|H|^2)^3 \\
&=&\!\!\!\Delta(|Rm|^2+|\nabla H|^2+|H|^4)+C[(|Rm|+|\nabla H|+|H|^2)^2]^{3/2} \\
&\le&\!\!\!\Delta(|Rm|^2+|\nabla H|^2+|H|^4)+C(|Rm|^2+|\nabla
H|^2+|H|^4)^{3/2}.
\end{eqnarray*}
$C,C_1,C_2$ above are all constants depending only on $n$. Whenever the supremum
is finite, we set
\[
K(t)=\sup_{x\in M^n}(|Rm(x,t)|^2+|\nabla H(x,t)|^2+|H(x,t)|^4).
\]
By the maximum principle, we have
\[
\frac{dK}{dt}\le CK^{\frac{3}{2}},
\]
which implies
\[
\frac{d}{dt}K^{-\frac{1}{2}}\ge-\frac{C}{2}.
\]
Integrating this inequality from $t$ to $\tau\in(t,T)$ and using the fact that
\[
\liminf_{\tau\to T}K(\tau)^{-\frac{1}{2}}=0,
\]
we obtain
\[
K(t)^{-\frac{1}{2}}\le\frac{C}{2}(T-t).
\]
Hence
\[
\sup_{x\in M^n}(|Rm(x,t)|^2+|\nabla H(x,t)|^2+|H(x,t)|^4)^{1/2}
\ge\frac{2}{C(T-t)}.
\]
Then as $P\ge(|Rm|^2+|\nabla H|^2+|H|^4)^{1/2}$, we finally get
\[
\sup_{x\in M^n}P(x,t)\ge\frac{c_0}{T-t}.
\]
\end{proof}

According to this result, we can classify the singular solutions of connection Ricci flow.

\begin{definition}
Define the following three types of singularities of connection Ricci flow
($T$ is the maximum time),

\textbf{Type} {\rm \textbf{\uppercase\expandafter{\romannumeral 1}}} \textbf{singularity:}
\[
T<\infty,\quad\sup_{M^n\times[0,T)}P\cdot(T-t)<\infty,
\]

\textbf{Type} {\rm \textbf{\uppercase\expandafter{\romannumeral 2}a}} \textbf{singularity:}
\[
T<\infty,\quad\sup_{M^n\times[0,T)}P\cdot(T-t)=\infty,
\]

\textbf{Type} {\rm \textbf{\uppercase\expandafter{\romannumeral 2}b}} \textbf{singularity:}
\[
T=\infty,\quad\sup_{M^n\times[0,\infty)}P\cdot t=\infty,
\]

\textbf{Type} {\rm \textbf{\uppercase\expandafter{\romannumeral 3}}} \textbf{singularity:}
\[
T=\infty,\quad\sup_{M^n\times[0,\infty)}P\cdot t<\infty.
\]
\end{definition}

Similar to Ricci flow, we can then give the definition of three singularity models.

\begin{definition}
A solution $(M^n,g(t),H(t))$ to the connection Ricci flow (1.1),
where either $M^n$ is compact, or at each time $t$, the metric $g(\cdot,t)$ is
complete and has bounded AC curvature, is called a \textbf{singularity model} if
it is not flat in the sense of AC curvature ($P\not\equiv 0$) and of one of the following
three types:

\textbf{Type} {\rm \textbf{\uppercase\expandafter{\romannumeral 1}}} \textbf{singularity model:}
The solution exists for $t\in(-\infty,\omega)$ for some constant $\omega$ with
$0<\omega<\infty$ and for any $x\in M^n$, any $t\in(-\infty,\omega)$,
\[
P(x,t)\le\frac{\omega}{\omega-t},
\]
with equality at $t=0$ and a point $y\in M^n$.

\textbf{Type} {\rm \textbf{\uppercase\expandafter{\romannumeral 2}}} \textbf{singularity model:}
The solution exists for $t\in(-\infty,+\infty)$, and for any $x\in M^n$, any
$t\in(-\infty,\omega)$,
\[
P(x,t)\le1,
\]
with equality at $t=0$ and a point $y\in M^n$.

\textbf{Type} {\rm \textbf{\uppercase\expandafter{\romannumeral 3}}} \textbf{singularity model:}
The solution exists for $t\in(-a,+\infty)$, for some constant $a$ with
$0<a<\infty$ and for any $x\in M^n$, any $t\in(-\infty,\omega)$,
\[
P(x,t)\le\frac{a}{a+t},
\]
with equality at $t=0$ and a point $y\in M^n$.
\end{definition}

In order to use compactness theorem (Theorem 2.4), we then define the injectivity
radius estimate.

\begin{definition}
A solution $(M^n,g(t),H(t))$ to the connection Ricci flow on the time interval
$[0,T)$ is said to satisfy an \textbf{injectivity radius estimate} if there exists a
constant $c_I>0$ such that
\[
{\rm inj}(x,t)^2\ge\frac{c_I}{\sup_{M^n}P(\cdot,t)}
\]
for all $(x,t)\in M^n\times[0,T)$.
\end{definition}

Following is the main result of this section.

\begin{theorem}
For any maximal solution to the connection Ricci flow which satisfies the
injectivity radius estimate and is of Type {\rm \uppercase\expandafter{\romannumeral 1}},
{\rm \uppercase\expandafter{\romannumeral 2}a}, {\rm \uppercase\expandafter{\romannumeral 2}b},
or {\rm \uppercase\expandafter{\romannumeral 3}}, there exists a sequence of dilations of
the solution which converges in the $C^\infty_{\rm loc}$ topology to a singularity model
of the corresponding type.
\end{theorem}

We will prove this theorem in next subsection.

\subsection{Convergence of the dilated solutions to connection Ricci flow}

Consider the solution $(M^n,g(t),H(t))$ to the connection Ricci flow (1.1), where either
$M^n$ is compact, or at each time $t$, the metric $g(\cdot,t)$ is complete and has bounded
AC curvature. To dilate, we choose $(x_i,t_i)$ such that
\begin{equation}
\sup_{M^n}P(\cdot,t_i)\ge c_1\sup_{M^n\times[s_i,t_i]}P,
\end{equation}
and
\begin{equation}
P(x_i,t_i)\ge c_2\sup_{M^n}P(\cdot,t_i),
\end{equation}
in which $c_1,c_2\in(0,1]$, and $s_i$ satisfies $(t_i-s_i)\sup_{M^n}P(\cdot,t_i)\to\infty$
(usually we set $s_i=0$).

Once we have chosen a sequence $\{(x_i,t_i)\}$ such that $t_i\nearrow T\in(0,\infty]$,
consider the solutions $(M^n,g_i(t),H_i(t))$ defined below
\begin{eqnarray}
g_i(t)\!\!\!&=&\!\!\!P(x_i,t_i)\cdot
g\Big(t_i+\frac{t}{P(x_i,t_i)}\Big),
\nonumber \\
H_i(t)\!\!\!&=&\!\!\!P(x_i,t_i)\cdot
H\Big(t_i+\frac{t}{P(x_i,t_i)}\Big),
\end{eqnarray}
in the time interval
\begin{equation}
-t_iP(x_i,t_i)\le t<(T-t_i)P(x_i,t_i).
\end{equation}
(The right endpoint is $\infty$ if $T=\infty$.) It can be easily verified that after this
dilation, they are still solutions to connection Ricci flow. And the AC curvature of the new
metric $g_i$ and 3-form $H_i$ has norm 1 at the point $x_i$ and the new time 0:
\begin{equation}
P[g_i,H_i](x_i,0)|_{g_i}=1.
\end{equation}
This guarantees that the limit of the pointed solutions $(M^n,g_i(t),H_i(t),x_i)$, if it
exists, will not be flat in the sense of AC curvature.

Assuming (3.2) and (3.3), one has the uniform AC curvature bound
\begin{equation}
P[g_i,H_i](x,t)\le\frac{1}{c_1c_2}
\end{equation}
for all $x\in M^n$ and $t\in[-(t_i-s_i)P(x_i,t_i),0]$. Combined with the above
injectivity radius estimate, by Theorem 2.4, there exists a subsequence of the
pointed sequence $(M^n,g_i(t),H_i(t),x_i)$ converges to a complete pointed solution $(M^n_\infty,g_\infty(t),H_\infty(t),x_\infty)$ to connection Ricci flow on an
ancient time interval $-\infty<t<\tau\le\infty$. This solution satisfies
\[
P[g_\infty,H_\infty](x,t)|_{g_\infty}\le\frac{1}{c_1c_2}
\]
for all $(x,t)\in M^n_\infty\times(-\infty,0]$.

\textbf{Limits of Type \uppercase\expandafter{\romannumeral 1} Singularities.}
Given a Type \uppercase\expandafter{\romannumeral 1} singular solution $(M^n,g(t),H(t))$
on $[0,T)$, define
\[
\Omega\triangleq\sup_{M^n\times[0,T)}P(x,t)\cdot(T-t)<\infty,
\]
and
\[
\omega\triangleq\limsup_{t\to T}P(\cdot,t)\cdot(T-t).
\]
By Proposition 3.4, we have $\omega\in[c_0,\Omega]\subset(0,\infty)$.

Taking a sequence of points and times $(x_i,t_i)$ with $t_i\nearrow T$ such that
\begin{equation}
P(x_i,t_i)\cdot(T-t_i)\triangleq\omega_i\to\omega.
\end{equation}
Consider the dilated solutions $(g_i(t),H_i(t))$ to connection Ricci flow defined
by (3.4). By the definition of $\omega$, there is for every $\epsilon>0$ a time
$t_\epsilon\in[0,T)$ such that
\[
P(x,t)\cdot(T-t)\le\omega+\epsilon
\]
for all $x\in M^n$ and $t\in[t_\epsilon,T)$. The AC curvature norm of $g_i$ and $H_i$
then satisfies
\begin{eqnarray}
P_i(x,t)\!\!\!&=&\!\!\!\frac{1}{P(x_i,t_i)}P\Big(x,t_i+\frac{t}{P(x_i,t_i)}\Big)
\nonumber \\
\!\!\!&=&\!\!\!\frac{P(x,t_i+\frac{t}{P(x_i,t_i)})\cdot(T-t_i-\frac{t}{P(x_i,t_i)})}
{P(x_i,t_i)\cdot(T-t_i)-t}
\nonumber \\
\!\!\!&\le&\!\!\!\frac{\omega+\epsilon}{\omega_i-t},
\end{eqnarray}
if $t_\epsilon\le t_i+t\cdot P(x_i,t_i)^{-1}<T$, hence if $t\in[-P(x_i,t_i)\cdot
(t_i-t_\epsilon),\omega_i)$.

Note that $\omega_i\to\omega$ and that $\lim_{i\to\infty}P(x_i,t_i)\cdot (t_i-t_\epsilon)
=\infty$ for any $\epsilon>0$. So if a pointed limit solution $(M^n_\infty,g_\infty(t),H_\infty(t),x_\infty)$ of a subsequence $\{M^n,g_i(t),H_i(t)\}$
exists, we can let $i\to\infty$ and then $\epsilon\to0$ in estimate (3.9) to conclude
that
\[
P_\infty(x,t)\le\frac{\omega}{\omega-t}
\]
for all $x\in M^n_\infty$ and $t\in(-\infty,\omega)$. Because $P_i(x_i,0)=1$ for all
$i$, the limit satisfies $P_\infty(x_\infty,0)=1$.

\textbf{Limits of Type \uppercase\expandafter{\romannumeral 2}a Singularities.}
Given a Type \uppercase\expandafter{\romannumeral 2}a singular solution
$(M^n,g(t),H(t))$ to connection Ricci flow on $[0,T)$, first let $T_i$ and $c'_i>0$
such that $T_i\nearrow T,\;c'_i\nearrow 1$ as $i\to\infty$. Then take a sequence
$\{(x_i,t_i)\}$, such that $t_i\nearrow T$, and
\begin{equation}
P(x_i,t_i)\cdot(T_i-t_i)\ge
c'_i\sup_{M^n\times[0,T_i]}P(x,t)\cdot(T_i-t).
\end{equation}

Consider the dilated solutions $(g_i(t),H_i(t))$ to connection Ricci flow defined
by (3.4). (3.10) implies that for $t\in[-t_iP(x_i,t_i),(T_i-t_i)P(x_i,t_i))$ we have
\begin{eqnarray*}
&&\!\!\!P_i(x,t)\Big(T_i-t_i-\frac{t}{P(x_i,t_i)}\Big) \\
&=&\!\!\!\frac{P(x,t_i+\frac{t}{P(x_i,t_i)})(T_i-t_i-\frac{t}{P(x_i,t_i)})}
{P(x_i,t_i)(T_i-t_i)}(T_i-t_i) \\
&\le&\!\!\!\frac{1}{c'_i}(T_i-t_i),
\end{eqnarray*}
or equivalently,
\[
P_i(x,t)\le\frac{1}{c'_i}\frac{(T_i-t_i)P(x_i,t_i)}{(T_i-t_i)P(x_i,t_i)-t}.
\]

Since $T_i\to T$ and
\[
\lim_{i\to\infty}[\sup_{M^n\times[0,T_i]}P(x,t)\cdot(T_i-t)]=\infty
\]
by the condition of Type \uppercase\expandafter{\romannumeral 2}a singularities,
so we have
\[
P(x_i,t_i)\cdot(T_i-t_i)\to\infty
\]
for the selected points and times $(x_i,t_i)$. Hence the pointed limit $(M^n_\infty,g_\infty(t),H_\infty(t),x_\infty)$, if it exists, is defined for all
$t\in(-\infty,\infty)$ and satisfies the uniform AC curvature bound
\[
\sup_{M^n_\infty\times(-\infty,\infty)}P_\infty\le1,
\]
with equality at $(x_\infty,0)$ because $P_i(x_i,0)=1$.

\textbf{Limits of Type \uppercase\expandafter{\romannumeral 2}b Singularities.}
The condition for this type of singularities is
\begin{equation}
\sup_{M^n\times[0,\infty)}t\cdot P(x,t)=\infty.
\end{equation}
Similar to the Type \uppercase\expandafter{\romannumeral 2}a case, let $T_j\to\infty$
and choose $(x_i,t_i)$ such that
\begin{equation}
\frac{t_i(T_i-t_i)\cdot
P(x_i,t_i)}{\sup_{M^n\times[0,T_i]}[t(T_i-t)\cdot
P(x,t)]}\triangleq1-\delta_i\to1.
\end{equation}

Define
\begin{eqnarray*}
\alpha_i\!\!\!&\triangleq&\!\!\!t_i\cdot P(x_i,t_i), \\
\omega'_i\!\!\!&\triangleq&\!\!\!(T_i-t_i)\cdot P(x_i,t_i).
\end{eqnarray*}
(3.11) and (3.12) guarantee that $\alpha_i\to\infty$ and $\omega'_i\to\infty$.
This is because
\begin{eqnarray*}
\frac{1}{\alpha_i^{-1}+\omega_i'^{-1}}\!\!\!&=&\!\!\!\frac
{\alpha_i\omega'_i}{\alpha_i+\omega'_i}=\frac{t_i\cdot
P(x_i,t_i)\cdot(T_i-t_i)}{T_i}
\\
\!\!\!&=&\!\!\!\frac{1}{T_i}\sup_{M^n\times[0,T_i]}t\cdot
P(x,t)\cdot(T_i-t) \\
\!\!\!&\ge&\!\!\!\frac{1}{2}\sup_{M^n\times[0,\frac{T_i}{2}]}t\cdot
P(x,t)\to\infty,
\end{eqnarray*}
hence $\alpha_i^{-1}\to0$ and $\omega_i'^{-1}\to0$. Meanwhile, for all $x\in M^n$
and $t\in[-\alpha_i,\omega_i')$, we have
\begin{eqnarray*}
P_i(x,t)\!\!\!&=&\!\!\!\frac{1}{P(x_i,t_i)}P\Big(x,t_i+\frac{t}{P(x_i,t_i)}\Big)
\\
\!\!\!&=&\!\!\!\frac{(t_i+\frac{t}{P(x_i,t_i)})(T_i-t_i-\frac{t}{P(x_i,t_i)})P(x,t_i+\frac{t}{P(x_i,t_i)})}
{t_i(T_i-t_i)P(x_i,t_i)} \\
&&\!\!\!\times\frac{t_i\cdot p(x_i,t_i)(T_i-t_i)}{(t_i\cdot
P(x_i,t_i)+t)(T_i-t_i-\frac{t}{P(x_i,t_i)})} \\
&\le&\!\!\!\frac{1}{1-\delta_i}\times\frac{\alpha_i}{\alpha_i+t}\frac{\omega_i'}{\omega_i'-t}.
\end{eqnarray*}
Since $\delta_i\to0,\;\alpha_i\to\infty$ and $\omega_i'\to\infty$, we conclude
that the pointed limit solution $(M^n_\infty,g_\infty(t),H_\infty(t),x_\infty)$,
if it exists, is defined for all $t\in(-\infty,\infty)$ and satisfies the AC curvature
bound
\[
\sup_{M^n_\infty\times(-\infty,\infty)}P_\infty\le1=P_\infty(x_\infty,0).
\]

\textbf{Limits of Type \uppercase\expandafter{\romannumeral 3} Singularities.}
If the solution $(M^n,g(t),H(t))$ to connection Ricci flow is a Type
\uppercase\expandafter{\romannumeral 3} singularity, then it exists for $t\in[0,\infty)$
and satisfies
\[
\sup_{M^n\times[0,\infty)}t\cdot P(\cdot,t)<\infty.
\]

Define
\[
a\triangleq\limsup_{t\to\infty}(t\cdot\sup_{M^n}P(\cdot,t))\in[0,\infty).
\]
First we need to show that $a$ is strictly positive. By the conclusion in Riemann geometry,
for a fixed curve $\gamma$ joining two points $p_1,p_2\in M^n$, under connection Ricci flow,
the length of the curve evolves like
\[
\frac{d}{dt}L_t(\gamma)=\frac{1}{2}\int_\gamma\frac{\partial
g}{\partial t}(\dot{\gamma},\dot{\gamma})ds=-\int_\gamma\Big({\rm
Rc}-\frac{1}{4}\mathcal{H}\Big)(\dot{\gamma},\dot{\gamma})ds.
\]
Next we argue by contradiction. Assume $a=0$, then we have

\textbf{Claim.} If there is an $\epsilon=\epsilon(n)>0$ small enough such that
\begin{equation}
a=\limsup_{t\to\infty}(t\cdot\sup_{M^n}P(\cdot,t))\le\epsilon,
\end{equation}
then there exist $C<\infty,\;\delta>0$ and $T_\epsilon<\infty$ depending only on
$n$ such that
\begin{equation}
{\rm diam}(M^n,g(t))\le Ct^{\frac{1}{2}-\delta}
\end{equation}
for all $t\ge T_\epsilon$.

To prove this claim, let $\gamma:[a,b]\to M^n$ be a fixed path, then
\[
\Big|\frac{dL}{dt}\Big|\le\int_\gamma\Big|{\rm
Rc}-\frac{1}{4}\mathcal{H}\Big|_{g(t)}ds.
\]
If (3.13) holds for some $\epsilon>0$, there exists a time $T_\epsilon<\infty$
such that for all $t>T_\epsilon$, we have
\[
t\cdot\sup_{M^n}P(\cdot,t)\le2\epsilon.
\]
In particular,
\[
\sup_{M^n}\Big|{\rm
Rc}-\frac{1}{4}\mathcal{H}(\cdot,t)\Big|\le\frac{C\epsilon}{t}
\]
for all $t\ge T_\epsilon$, where $C$ depends only on $n$. Hence
\[
\frac{dL}{dt}(\tau)\le\Big|\frac{dL}{dt}(\tau)\Big|\le\frac{C\epsilon}{\tau}L(\tau)
\]
for $\tau\ge T_\epsilon$. Integrating this inequality from time $T_\epsilon$ to time
$t>T_\epsilon$ implies
\[
L(t)\le
L(T_\epsilon)\Big(\frac{t}{T_\epsilon}\Big)^{C\epsilon}=L(T_\epsilon)T_\epsilon^{-C\epsilon}\cdot
t^{C\epsilon}.
\]
In particular, if we choose $0<\epsilon<\frac{1}{2C}$, then any two points in $(M^n,g(t))$
can be joined by a path of length
\[
L(t)\le{\rm diam}(M^n,g(T_\epsilon))\cdot
T_\epsilon^{-C\epsilon}\cdot t^{\frac{1}{2}-\delta},
\]
where $\delta\triangleq\frac{1}{2}-C\epsilon$. This implies (3.14) and proves the claim.

Using the claim, let $\Omega'\triangleq\sup_{M^n\times[0,\infty)}(t\cdot\sup_{M^n}P(\cdot,t))<\infty$,
then for any $t\in(0,\infty)$,
\[
\sup_{M^n}P(\cdot,t)\cdot{\rm
diam}(M^n,g(t))^2\le\Omega'C^2t^{-2\delta}.
\]
Since $\Omega'C^2t^{-2\delta}\to0$ as $t\to\infty$, this contradicts the injectivity radius
estimate! Hence our original assumption is false, that is $a>0$.

By the definition of $a$, there exist sequences $(x_i,t_i)$ with $t_i\to\infty$ such that
\[
a_i\triangleq t_i\cdot P(x_i,t_i)\to a.
\]
Choose any such sequence. Also by the definition of $a$, there is for any $\xi>0$ a
time $T_\xi\in[0,\infty)$ such that
\[
t\cdot P(x,t)\le a+\xi
\]
for all $x\in M^n$ and $t\in[T_\xi,\infty)$. The dilated solutions $(g_i(t),H_i(t))$
exist on the time intervals $[-a_i,\infty)$ and satisfy
\begin{eqnarray*}
P_i(x,t)\!\!\!&=&\!\!\!\frac{1}{P(x_i,t_i)}P\Big(x,t_i+\frac{t}{P(x_,t_i)}\Big)
\\
&=&\!\!\!\frac{P(x,t_i+\frac{t}{P(x_,t_i)})\cdot(t_i+\frac{t}{P(x_,t_i)})}{P(x_i,t_i)\cdot
t_i+t} \\
&\le&\!\!\!\frac{a+\xi}{a_i+t},
\end{eqnarray*}
if $t_i+t\cdot P(x_i,t_i)^{-1}\ge T_\xi$, that is if $t\ge P(x_i,t_i)\cdot(T_\xi-t_i)$.
Note that for any fixed $\xi>0$, we have
\[
P(x_i,t_i)\cdot(T_\xi-t_i)\to-a
\]
and
\[
\frac{a+\xi}{a_i+t}\to\frac{a+\xi}{a+t}
\]
for $t>-a$. Hence the pointed limit solution $(M^n_\infty,g_\infty(t),H_\infty(t),x_\infty)$,
if it exists, is defined for $t\in(-a,\infty)$ and satisfies
\[
\sup_{M^n_\infty\times(-a,\infty)}P_\infty\le\frac{a}{a+t}=P_\infty(x_\infty,0).
\]

To sum up, we finish the proof of Theorem 3.8.

\section{Singularities of Ricci harmonic flow}

\subsection{Singularity models and convergence of dilated solutions}

By Theorem 2.8, if Ricci harmonic flow has finite-time
singularities, then
\[
\lim_{t\to T}\sup_{x\in M^m}|Rm(x,t)|=\infty,
\]
where $T$ is the maximum time.

Similar to the case of connection Ricci flow, define an AC curvature
\[
Q=|Rm|+|\nabla^2\phi|+|\nabla\phi|^2.
\]

Similarly, if Ricci harmonic flow has finite-time singularities,
then
\[
\lim_{t\to T}\sup_{x\in M^m}Q(x,t)=\infty.
\]
Besides, here we also have

\begin{proposition}
If $0\le t<T<\infty$ is
the maximal interval of existence of the solution
$(M^m,g(t),\phi(t))$ to Ricci harmonic flow, there exists a constant
$c_0>0$ depending only on $m,\alpha$ and the curvature of manifold
$N^n$ such that
\[
\sup_{x\in M^m}Q(x,t)\ge\frac{c_0}{T-t}.
\]
\end{proposition}

\begin{proof}
By Proposition 2.5,
\begin{eqnarray*}
\frac{\partial}{\partial
t}|Rm|^2\!\!\!&\le&\!\!\!\Delta|Rm|^2-2|\nabla
Rm|^2+C|Rm|^2|\nabla\phi|^2+C|Rm||\nabla^2\phi|^2 \\
&&\!\!\!+C|Rm||\nabla\phi|^4+C|Rm|^3, \\
\frac{\partial}{\partial t}|\nabla^2\phi|^2\!\!\!&\le&\!\!\!
\Delta|\nabla^2\phi|^2-2|\nabla^3\phi|^2+C|Rm||\nabla^2\phi|^2
+C|\nabla^2\phi|^2|\nabla\phi|^2 \\
&&\!\!\!+C|\nabla^2\phi||\nabla\phi|^4, \\
\frac{\partial}{\partial t}|\nabla\phi|^4\!\!\!&\le&\!\!\!
\Delta|\nabla\phi|^4+C|\nabla\phi|^6.
\end{eqnarray*}
After computation, we can also get
\begin{eqnarray*}
&&\!\!\!\frac{\partial}{\partial t}(|Rm|^2+|\nabla^2\phi|^2+|\nabla\phi|^4) \\
&\le&\!\!\!\Delta(|Rm|^2+|\nabla^2\phi|^2+|\nabla\phi|^4)
+C(|Rm|^2+|\nabla^2\phi|^2+|\nabla\phi|^4)^{3/2}
\end{eqnarray*}
for a constant $C>0$ depending only on $m,\alpha$ and the curvature of manifold
$N^n$. The following is like that in the proof of Proposition 3.4.
\end{proof}

Replace the $P$ in the case of connection Ricci flow by $Q$, we can then obtain
the classification of singular solutions, singularity models and convergence of
dilated solutions of Ricci harmonic flow.

\begin{definition}
Define the following three types of singularities of Ricci harmonic flow
($T$ is the maximum time),

\textbf{Type} {\rm \textbf{\uppercase\expandafter{\romannumeral 1}}} \textbf{singularity:}
\[
T<\infty,\quad\sup_{M^m\times[0,T)}Q\cdot(T-t)<\infty,
\]

\textbf{Type} {\rm \textbf{\uppercase\expandafter{\romannumeral 2}a}} \textbf{singularity:}
\[
T<\infty,\quad\sup_{M^m\times[0,T)}Q\cdot(T-t)=\infty,
\]

\textbf{Type} {\rm \textbf{\uppercase\expandafter{\romannumeral 2}b}} \textbf{singularity:}
\[
T=\infty,\quad\sup_{M^m\times[0,\infty)}Q\cdot t=\infty,
\]

\textbf{Type} {\rm \textbf{\uppercase\expandafter{\romannumeral 3}}} \textbf{singularity:}
\[
T=\infty,\quad\sup_{M^m\times[0,\infty)}Q\cdot t<\infty.
\]
\end{definition}

Then the definition of the corresponding three types of singularity models

\begin{definition}
A solution $(M^m,g(t),\phi(t))$ to the Ricci harmonic flow (1.2),
where either $M^m$ is compact, or at each time $t$, the metric $g(\cdot,t)$ is
complete and has bounded AC curvature, is called a \textbf{singularity model} if
it is not flat in the sense of AC curvature ($Q\not\equiv 0$) and of one of the following
three types:

\textbf{Type} {\rm \textbf{\uppercase\expandafter{\romannumeral 1}}} \textbf{singularity model:}
The solution exists for $t\in(-\infty,\omega)$ for some constant $\omega$ with
$0<\omega<\infty$ and for any $x\in M^m$, any $t\in(-\infty,\omega)$,
\[
Q(x,t)\le\frac{\omega}{\omega-t},
\]
with equality at $t=0$ and a point $y\in M^m$.

\textbf{Type} {\rm \textbf{\uppercase\expandafter{\romannumeral 2}}} \textbf{singularity model:}
The solution exists for $t\in(-\infty,+\infty)$, and for any $x\in M^m$, any
$t\in(-\infty,\omega)$,
\[
Q(x,t)\le1,
\]
with equality at $t=0$ and a point $y\in M^m$.

\textbf{Type} {\rm \textbf{\uppercase\expandafter{\romannumeral 3}}} \textbf{singularity model:}
The solution exists for $t\in(-a,+\infty)$, for some constant $a$ with
$0<a<\infty$ and for any $x\in M^m$, any $t\in(-\infty,\omega)$,
\[
Q(x,t)\le\frac{a}{a+t},
\]
with equality at $t=0$ and a point $y\in M^m$.
\end{definition}

The injectivity radius estimate for Ricci harmonic flow,

\begin{definition}
A solution $(M^m,g(t),\phi(t))$ to Ricci harmonic flow on the time interval
$[0,T)$ is said to satisfy an \textbf{injectivity radius estimate} if there exists a
constant $c_I>0$ such that
\[
{\rm inj}(x,t)^2\ge\frac{c_I}{\sup_{M^m}Q(\cdot,t)}.
\]
for all $(x,t)\in M^m\times[0,T)$.
\end{definition}

Then we give the convergence of dilated solutions.

\begin{theorem}
For any maximal solution to Ricci harmonic flow which satisfies the injectivity
radius estimate and is of Type {\rm \uppercase\expandafter{\romannumeral 1}},
{\rm \uppercase\expandafter{\romannumeral 2}a}, {\rm \uppercase\expandafter{\romannumeral 2}b},
or {\rm \uppercase\expandafter{\romannumeral 3}}, there exists a sequence of dilations of
the solution which converges in the $C^\infty_{\rm loc}$ topology to a singularity model
of the corresponding type.
\end{theorem}

For Ricci harmonic flow, the dilated solutions are defined as
\begin{eqnarray}
g_i(t)\!\!\!&=&\!\!\!Q(x_i,t_i)\cdot
g\Big(t_i+\frac{t}{Q(x_i,t_i)}\Big),
\nonumber \\
\phi_i(t)\!\!\!&=&\!\!\!\phi\Big(t_i+\frac{t}{Q(x_i,t_i)}\Big),
\end{eqnarray}
And note that we have established the compactness theorem (Theorem 2.9) for Ricci harmonic 
flow. The rest of the proof is the same as that in connection Ricci flow.

\subsection{Compact blow-up limits of finite-time singularities are shrinking solitons}

Analogous to the work for Ricci flow in [Zhang], we can connect the
finite-time singularities of Ricci harmonic flow with shrinking
Ricci harmonic solitons.

\begin{theorem}
Let
$(g(t),\phi(t))_{t\in[0,T)}$, be a maximal solution to the Ricci
harmonic flow (1.2) on a closed manifold $M^m$ with singular time
$T<\infty$. Let $t_k\to T$ be a sequence of times such that
$Q_k=Q(p_k,t_k)\to\infty$. If the rescaled sequence
$(M^m,Q_kg(t_k+Q_k^{-1}t),\phi(t_k+Q_k^{-1}t),p_k)$ converges in the
$C^\infty$ sense to a closed ancient solution
$(M^m_\infty,g_\infty(t),\phi_\infty(t),p_{\infty})$ to the Ricci
harmonic flow, then $(g_\infty(t),\phi_\infty(t))$ must be a
shrinking Ricci harmonic soliton.
\end{theorem}

First we derive some estimates about $\lambda_\alpha$ and
$\mu_\alpha$ functionals.

\begin{lemma}
We have the upper bound
\begin{equation}
\mu_\alpha(g,\phi,\tau)\le\tau\lambda_\alpha(g,\phi)+{\rm
Vol}(g)-\frac{m}{2} \ln(4\pi\tau)-m,
\end{equation}
and the lower bound for $\tau>\frac{m}{8}$,
\begin{equation}
\mu_\alpha(g,\phi,\tau)\ge\tau\lambda_\alpha-\frac{m}{2}\ln(4\pi\tau)-m-\frac{m}{8}
(\lambda_\alpha-\inf\{R-\alpha|\nabla\phi|^2\})-m\ln C_s,
\end{equation}
where $C_s$ denotes the Sobolev constant for $g$ such that
$\|\psi\|_{L^ {\frac{2m}{m-2}}(g)}\le C_s\|\psi\|_{H^{1,2}(g)}$ for
all $\psi\in C^{\infty}(M^m)$.
\end{lemma}

\begin{proof}
By definition, set
$u=(4\pi\tau)^{-m/4}e^{-f/2}$, then $\int_Mu^2dV=1$ and so
\begin{eqnarray*}
&&\!\!\!\int_{M^m}[\tau(R+|\nabla f|^2-\alpha|\nabla\phi|^2)+f-m](4\pi\tau)^{-m/2}e^{-f}dV \\
&=&\!\!\!\tau\int_{M^m}(Ru^2+4|\nabla
u|^2-\alpha|\nabla\phi|^2u^2)dV-\int_{M^m}u^2\ln u^2dV
-\frac{m}{2}\ln(4\pi\tau)-m \\
&\le&\!\!\!\tau\int_{M^m}(Ru^2+4|\nabla
u|^2-\alpha|\nabla\phi|^2u^2)dV+{\rm Vol}(g)
-\frac{m}{2}\ln(4\pi\tau)-m,
\end{eqnarray*}
where we used that $-t\ln t\le1$ for all $\tau>0$. The upper bound
follows by choosing $f$ such that $u$ is the eigenfunction of the
first eigenvalue of $-4\Delta+R-\alpha|\nabla\phi|^2$.

As for the lower bound, let $\bar f$ be the minimizer of
$\mu_\alpha(g,\phi,\tau)$ for fixed $\tau>0$ and set $\bar
u=(4\pi\tau)^{-m/4}e^{-\bar f/2}$. We estimate the term
$-\int_{M^m}\bar u^2\ln\bar u^2dV$ by
\begin{eqnarray*}
-\int_{M^m}\bar u^2\ln\bar
u^2dV\!\!\!&=&\!\!\!-\frac{m-2}{2}\int_{M^m}\bar u^2\ln\bar u^{
\frac{4}{m-2}}dV\ge-m\ln\|\bar u\|_{L^{\frac{2m}{m-2}}(g)} \\
\!\!\!&\ge&\!\!\!-\frac{m}{2}\ln\bigg(1+\int_{M^m}|\nabla\bar
u|^2dV\bigg)-m\ln C_s,
\end{eqnarray*}
where we used the Jensen and Sobolev inequality in the first and the
second inequality. Then we have
\begin{eqnarray*}
&&\!\!\!\mu_\alpha(g,\phi,\tau) \\
&=&\!\!\!\tau\int_{M^m}(R\bar u^2+4|\nabla\bar
u|^2-\alpha|\nabla\phi|^2\bar u^2)dV
-\int_{M^m}\bar u^2\ln\bar u^2dV-\frac{m}{2}\ln(4\pi\tau)-m \\
&\ge&\!\!\!\tau\int_{M^m}(R\bar u^2+4|\nabla\bar
u|^2-\alpha|\nabla\phi|^2\bar
u^2)dV -\frac{m}{2}\ln\bigg(1+\int_{M^m}|\nabla\bar u|^2dV\bigg) \\
&&\!-\frac{m}{2}\ln(4\pi\tau)-m-m\ln C_s \\
&\ge&\!\!\!\bigg(\tau-\frac{m}{8}\bigg)\int_{M^m}(R\bar
u^2+4|\nabla\bar u|^2-\alpha |\nabla\phi|^2\bar
u^2)dV+\frac{m}{8}\int_{M^m}(R-\alpha|\nabla\phi|^2)\bar u^2dV \\
&&\!-\frac{m}{2}\ln(4\pi\tau)-m-m\ln C_s,
\end{eqnarray*}
which proves the lower bound if we set $\tau>\frac{m}{8}$.
\end{proof}

\begin{corollary}
\[
\nu_\alpha(g,\phi)=-\infty\iff\lambda_\alpha(g,\phi)\le0,\quad\nu_\alpha(g,\phi)
\ne-\infty\iff\lambda_\alpha(g,\phi)>0.
\]
\end{corollary}

\begin{proof}
From the conclusion of [Ro] which indicates
that the functional $\mu_\alpha(g,\phi,\tau)$ in Ricci harmonic flow
is always attainable by some smooth function, it can be easily seen
that $\mu_\alpha(g,\phi,\tau)$ is continuous for $\tau>0$.

When $\lambda_\alpha(g,\phi)\le0$, since
\[
\lim_{\tau\to\infty}\Big[\tau\lambda_\alpha(g,\phi)-\frac{m}{2}\ln(4\pi\tau)\Big]
=-\infty,
\]
from the first inequality of above lemma,
$\nu_\alpha(g,\phi)=-\infty$.

When $\lambda_\alpha(g,\phi)>0$, since
\[
\lim_{\tau\to\infty}\Big[\Big(\tau-\frac{m}{8}\Big)\lambda_\alpha(g,\phi)-\frac
{m}{2}\ln(4\pi\tau)\Big]=+\infty,
\]
from the second inequality of above lemma,
$\mu_\alpha(g,\phi,\tau)>-\infty$ as $\tau\to\infty$. Like the
functional $\mu(g,\tau)$ in Ricci flow,
$\mu_\alpha(g,\phi,\tau)>-\infty$ as $\tau\to0^+$. This gives
$\nu_\alpha(g,\phi) \ne-\infty$. Corollary is proved.
\end{proof}

It follows that $\nu_\alpha$ functional is valuable only when
$\lambda_\alpha>0$. The assumption of our main theorem implies the
positivity of $\lambda_\alpha$ along the Ricci harmonic flow. This
fact will be proved later.

\begin{corollary}
If
$\lambda_\alpha(g,\phi)\le0$, then
$\mu_\alpha(g,\phi,\tau)\le\ln{\rm Vol}
(g)-\frac{m}{2}\ln(4\pi\tau)-m+1$.
\end{corollary}

\begin{proof}
First note that
$\mu_\alpha(ag,\phi,a\tau)=\mu_\alpha(g,\phi,\tau)$ for any $a>0$ by
a direct computation. Set $V={\rm Vol}(g)^{-2/m}$, then by the above
lemma,
\begin{eqnarray*}
\mu_\alpha(g,\phi,\tau)\!\!\!&=&\!\!\!\mu_\alpha(Vg,\phi,V\tau) \\
\!\!\!&\le&\!\!\!V\tau\lambda_\alpha(Vg,\phi)-\frac{m}{2}\ln(4\pi V\tau)-m+{\rm Vol}(Vg) \\
\!\!\!&\le&\!\!\!\ln{\rm Vol}(g)-\frac{m}{2}\ln(4\pi\tau)-m+1.
\end{eqnarray*}
\end{proof}

\begin{lemma}
Let
$(g(t),\phi(t)),\,t\in[0,T)$, be a solution to the Ricci harmonic
flow on a closed manifold $M^m$. If
$\lambda_\alpha(g(t),\phi(t))\le0$ for all $t$, then there exist
constants $c_1,c_2>0$ depending only on $m,\alpha$ and $(g(0),
\phi(0))$, such that for all $t\ge0$ we have ${\rm Vol}(g(t))\ge c_1
{\rm e}^{-c_2t}$.
\end{lemma}

\begin{proof}
By Proposition 2.7 and Lemma 4.7, we have
\begin{eqnarray*}
&&\!\!\!\mu_\alpha\Big(g(t),\phi(t),\frac{m}{8}\Big)\ge\mu_\alpha\Big(g(0),
\phi(0),\frac{m}{8}+t\Big) \\
&\ge&\!\!\!\lambda_\alpha(g(0),\phi(0))t-\frac{m}{2}\ln\bigg(4\pi\bigg(t+
\frac{m}{8}\bigg)\bigg) \\
&&\!+\frac{m}{8}\inf\{R(\cdot,0)-\alpha|\nabla\phi|^2(\cdot,0)\}-m-m\ln
C_s(g(0)) \\
&\ge&\!\!\!\bigg(\lambda_\alpha(g(0),\phi(0))-4\bigg)t-\frac{m}{2}\ln\bigg
(\frac{m}{2}\pi\bigg) \\
&&\!+\frac{m}{8}\inf\{R(\cdot,0)-\alpha|\nabla\phi|^2(\cdot,0)\}-m-m\ln
C_s(g(0)),
\end{eqnarray*}
where $C_s(g(0))$ denotes the Sobolev constant of $(M^m,g(0))$.
Setting
\[
c_1=\exp\Big(\frac{m}{8}\inf\{R(\cdot,0)-\alpha|\nabla\phi|^2(\cdot,0)\}
-m\ln C_s(g(0))-1\Big),\quad c_2=-\lambda_\alpha(g(0),\phi(0))+4,
\]
and substituting $\tau=\frac{m}{8}$ into Corollary 4.9, we obtain
the estimate
\[{\rm
Vol}(g(t))\ge\exp\Big(\mu_\alpha\Big(g(t),\phi(t),\frac{m}{8}\Big)\Big)+\frac{m}{2}
\ln\Big(\frac{m}{2}\pi\Big)+m-1\ge c_1\exp(-c_2t).
\]
\end{proof}

\begin{corollary}
Let
$(g(t),\phi(t)),\,t\in[0,T)$, be a maximal solution to the Ricci
harmonic flow on a closed manifold $M^m$ with $T<\infty$. If
$\lambda_\alpha(g(t),\phi(t)) \le0$ for all $t$, then any blow-up
limit is noncompact.
\end{corollary}

\begin{proof}
Suppose we have a blow-up sequence
$(M^m,Q_kg(t_k+Q_k^{-1}t),\phi(t_k+Q_k^{-1}t),p_k)$ of Ricci
harmonic flow solutions with $Q_k\to\infty$. By assumption and above
lemma, we have that the rescaled volume at time zero equals
$Q_k^{m/2}{\rm Vol}(g(t_k))\to\infty$. So the limit has infinite
volume and consequently cannot be compact.
\end{proof}

Now we are ready to give a

\noindent\emph{Proof of Theorem\;4.6.} By above corollary, we may
assume that $\lambda_\alpha(g(0),\phi(0))>0$. So Proposition 2.7
uses and there is a limit $\sigma=\lim_{t\to
T^-}\nu_\alpha(g(t),\phi(t))$. Then for any $t\in(-\infty,0]$, by
the smooth convergence,
\begin{eqnarray*}
\nu_\alpha(g_\infty(t),\phi_\infty(t))\!\!\!&=&\!\!\!\lim_{k\to\infty}\nu_\alpha
(Q_kg(t_k+Q_k^{-1}t),\phi(t_k+Q_k^{-1}t)) \\
\!\!\!&=&\!\!\!\lim_{k\to\infty}\nu_\alpha(g(t_k+Q_k^{-1}t),\phi(t_k+Q_k^{-1}t)) \\
\!\!\!&=&\!\!\!\lim_{t\to T^-}\nu_\alpha(g(t),\phi(t))=\sigma.
\end{eqnarray*}
That is, the $\nu_\alpha$ functional is constant on the limit flow.
Then Corollary 4.8 and Proposition 2.7 imply that
$(g_\infty(t),\phi_\infty(t))$  must be a shrinking Ricci harmonic
soliton. \hfill$\Box$

\subsection{Ancient solutions of Ricci harmonic flow have nonnegative scalar curvature}

In this part, we study the ancient solutions of Ricci harmonic flow in the case of compact
and noncompact manifolds. Specifically, the Type \uppercase\expandafter{\romannumeral1}
singularity model given by \S4.2 is just an ancient solution.

First is the compact case.

\begin{theorem}
Let $(g(t),\phi(t))$ be an ancient solution to Ricci harmonic flow (1.2) on a compact
manifold $M^m$, then for any $t$ such that the solution exists, we have $S=R-\alpha|\nabla\phi|^2\ge0$.
That is, the scalar curvature is always nonnegative.
\end{theorem}

\begin{proof}
The main idea of the proof is using maximum principle to (2.14)
\[
\frac{\partial}{\partial t}S=\Delta S+2|S_{ij}|^2+2\alpha
|\tau_g\phi|^2.
\]

Define $S_{\min}(t)=\min_{M^m}S(\cdot,t)$, then
\begin{equation}
\frac{d}{dt}S_{\min}\ge\frac{2}{m}S_{\min}^2\ge0.
\end{equation}
Suppose there exists a time $t_0\in(-\infty,\Omega)$ and a point $x_0\in M^m$,
such that $S(x_0,t_0)<0$, then $S_{\min}(t_0)<0$. Pick some time $t_1<t_0$,
integrating (4.4) from $t_1$ to $t_0$ yields
\[
\frac{1}{S_{\min}(t_1)}\ge\frac{2}{m}(t_0-t_1)+\frac{1}{S_{\min}(t_0)}.
\]
Since $(g(t),\phi(t))$ is an ancient solution, $t_1$ can be chosen in the whole interval
$(-\infty,t_0)$. Let
\[
t_1=t_0+\frac{m}{2S_{\min}(t_0)},
\]
then we have
\[
\frac{1}{S_{\min}(t_1)}\ge0\;\Rightarrow\;S_{\min}(t_1)\ge0\ge
S_{\min}(t_0).
\]
But by (4.4), $S_{\min}$ increases, which is a contradiction! So the assumption is not
true, and the theorem is proved.
\end{proof}

As for the noncompact case, refer to the work for Ricci flow in [Chen], similar
conclusion can still be obtained.

Before giving the theorem, we first introduce the following lemma.

\begin{lemma}
Let $(g(x,t),\phi(x,t))$ be a solution to Ricci harmonic flow (1.2) on $M^m$,
and denote by $d_t(x,x_0)$ the distance between $x$ and $x_0$ with respect to the
metric $g(t)$. Suppose ${\rm
Rc}(\cdot,t)\le(m-1)K$ on $B_{t_0}(x_0,r_0)$ for some $x_0\in M^m$ and some positive
constant $K$ and $r_0$. Then at $t=t_0$ and outside $B_{t_0}(x_0,r_0)$, the distance
function $d(x,t)=d_t(x,x_0)$ satisfies the differential inequality
\[
\bigg(\frac{\partial}{\partial t}-\Delta\bigg)d\ge-(m-1)\bigg(\frac{2}{3}
Kr_0+r_0^{-1}\bigg).
\]
\end{lemma}

\begin{proof}
Let $\gamma:[0,d(x,x_0)]\to M^m$ be a shortest normal geodesic
from $x_0$ to $x$ with respect to the metric $g(t_0)$. We may assume that $x$ and $x_0$
are not conjugate to each other in the metric $g(t_0)$, otherwise we can understand the
differential inequality in the barrier sense. Let $X=\dot{\gamma}(0)$, and let
$\{X,e_1,\cdots,e_{m-1}\}$ be an orthonormal basis of $T_{x_0}M$. Extend this basis
parallel along $\gamma$ to form a parallel orthonormal basis $\{X(s),e_1(s),\cdots,e_{m-1}(s)\}$
along $\gamma$.

Let $X_i(s),\;i=1,\cdots,m-1$ be the Jacobian fields along $\gamma$ such that $X_i(0)=0$
and $X_i(d(x,t_0))=e_i(d(x,t_0))$ for $i=1,\cdots,m-1$. Then it is well-known that
\[
\Delta d_{t_0}(x,x_0)=\sum_{i=1}^{m-1}\int_0^{d(x,t_0)}(|\dot{X}_i|^2-
R(X,X_i,X,X_i))ds.
\]

Define vector fields $Y_i,\;i=1,\cdots,m-1$, along $\gamma$ as follows:
\[
Y_i(s)=
\left\{
\begin{array}{l l}
\frac{s}{r_0}e_i(s), & \mbox{if}s\in[0,r_0], \\
e_i(s), &  \mbox{if}s\in[r_0,d(x,t_0)].
\end{array}
\right.
\]
They have the same value as the corresponding Jacobian fields $X_i(s)$ at the two
end points of $\gamma$. Then by standard index comparison theorem we have
\begin{eqnarray*}
\Delta d_{t_0}(x,x_0)\!\!\!&=&\!\!\!\sum_{i=1}^{m-1}\int_0^{d(x,t_0)}
(|\dot{X}_i|^2-R(X,X_i,X,X_i))ds \\
\!\!\!&\le&\!\!\!\sum_{i=1}^{m-1}\int_0^{d(x,t_0)}(|\dot{Y}_i|^2-
R(X,Y_i,X,Y_i))ds \\
\!\!\!&=&\!\!\!\int_0^{r_0}\frac{1}{r_0^2}(m-1-s^2{\rm Rc}(X,X))ds
+\int_{r_0}^{d(x,t_0)}(-{\rm Rc}(X,X))ds \\
\!\!\!&=&\!\!\!-\int_\gamma{\rm Rc}(X,X)+\int_0^{r_0}\bigg(\frac{m-1}{r_0^2}
+\bigg(1-\frac{s^2}{r_0^2}\bigg){\rm Rc}(X,X)\bigg)ds \\
\!\!\!&\le&\!\!\!-\int_\gamma{\rm Rc}(X,X)+(m-1)\bigg(\frac{2}{3}Kr_0
+r_0^{-1}\bigg).
\end{eqnarray*}
On the other hand,
\begin{eqnarray*}
\frac{\partial}{\partial t}d_t(x,x_0)\!\!\!&=&\!\!\!\frac{\partial}{\partial t}
\int_0^{d(x,t_0)}\sqrt{g_{ij}X^iX^j}ds \\
\!\!\!&=&\!\!\!-\int_\gamma({\rm Rc}-\alpha\nabla\phi\otimes\nabla\phi)(X,X)ds.
\end{eqnarray*}
Make the subtraction and consider that $(\nabla\phi\otimes\nabla\phi)(X,X)=|\nabla_X\phi|^2\ge0$,
then the lemma is proved.
\end{proof}

Now we state and prove the result for noncompact case.

\begin{theorem}
Let $(g(t),\phi(t))$ be a complete ancient solution to Ricci harmonic flow (1.2)
on a noncompact manifold $M^m$, then for any $t$ such that the solution exists,
we have $S=R-\alpha|\nabla\phi|^2\ge0$. That is, the scalar curvature is always
nonnegative.
\end{theorem}

\begin{proof}
Suppose $(g(t),\phi(t))$ is defined for $t\in(-\infty,T]$ for some $T>0$.
We divide the arguments into two steps.

Step 1: Consider any complete solution $(g(t),\phi(t))$ defined on $[0,T]$.
For any fixed point $x_0\in M^m$, pick $r_0>0$ sufficiently small so that
\begin{equation}
|{\rm Rc}|(\cdot,t)\le(m-1)r_0^{-2}\quad\mbox{on}\;B_t(x_0,r_0)
\end{equation}
for all $t\in[0,T]$. Then for any positive number $A>2$, pick $K_A>0$ such that
$S\ge-K_A$ on $B_0(x_0,Ar_0)$ at $t=0$. We claim that there exists a universal
constant $C>0$ (depending on the dimension $m$) such that
\begin{equation}
S(\cdot,t)\ge\min\bigg\{-\frac{m}{t+K_A^{-1}},-\frac{C}{Ar_0^2}\bigg\}\quad
\mbox{on}\;B_t\bigg(x_0,\frac{3A}{4}r_0\bigg)
\end{equation}
for each $t\in[0,T]$.

Take a smooth nonnegative decreasing function $\bar{\psi}$ on
$\mathbb{R}$ such that $\bar{\psi}=1$ on $(-\infty,\frac{7}{8}]$,
and $\bar{\psi}=0$ on $[1,\infty)$. Also, let it satisfy $|\bar{\psi}'|\le
\bar{C}_1\bar{\psi}^{1/2},|\bar{\psi}''|\le\bar{C}_2$.
If we set $\psi=\bar{\psi}^p$, then we can find suitable $p>0$,
such that
\[
|\psi'|\le
C_1\psi^{\frac{1}{2}},\quad\Big|\frac{2\psi'^2}{\psi}-\psi''\Big|\le
C_2\psi^{\frac{1}{2}}.
\]
It can be verified that $\psi$ is still a smooth nonnegative decreasing
function on $\mathbb{R}$ which has value 1 on $(-\infty,\frac{7}{8}]$
and value 0 on $[1,\infty)$.

Consider the function
\[
u(x,t)=\psi\bigg(\frac{d_t(x_0,x)}{Ar_0}\bigg)\cdot S(x,t).
\]
Then we have
\begin{equation}
\bigg(\frac{\partial}{\partial t}-\Delta\bigg)u=\frac{\psi'S}{Ar_0}
\bigg(\frac{\partial}{\partial t}-\Delta\bigg)d_t(x_0,x)-\frac{\psi''S}{(Ar_0)^2}
+2\psi(|S_{ij}|^2+\alpha|\tau_g\phi|^2)-2\nabla\psi\cdot\nabla S
\end{equation}
at smooth points of the distance function $d_t(x_0,\cdot)$.

Let $u_{\min}(t)=\min_{M^m}u(\cdot,t)$. Whenever $u_{\min}(t_0)\le0$,
assume $u_{\min}(t_0)$ is achieved at some point $\bar{x}$, then at
$(\bar{x},t_0)$,
\[
\nabla(\psi S)=0,\;\Delta(\psi S)\ge0,\;\psi'S\ge0.
\]
When $\bar{x}$ is outside $B_{t_0}(x_0,r_0)$, by (4.5) and Lemma 4.13,
\[
\bigg(\frac{\partial}{\partial
t}-\Delta\bigg)d_t(x_0,x)\ge-\frac{5(m-1)}{3r_0}.
\]
Taking these inequalities into (4.7), we get
\begin{eqnarray*}
\frac{d}{dt}\bigg|_{t=t_0}u_{\min}\!\!\!&\ge&\!\!\!-\frac{5(m-1)}{3Ar_0^2}\psi'S+\frac{2}{m}\psi
S^2+\frac{1}{(Ar_0)^2}\bigg(\frac{2\psi'^2}{\psi}-\psi''\bigg)S \\
\!\!\!&\ge&\!\!\!\frac{2}{m}\psi
S^2+\frac{C'}{Ar_0^2}\psi^{\frac{1}{2}}S
\end{eqnarray*}
as long as $u_{\min}(t_0)\le0$. Here we use the definition
\[
\frac{d}{dt}\bigg|_{t=t_0}u_{\min}=\liminf_{h\to0^+}\frac{u_{\min}(t_0+h)-u_{\min}(t_0)}{h},
\]
and the prperty of $\psi$. Then by Cauchy inequality,
\[
\frac{1}{Ar_0^2}\psi^{\frac{1}{2}}S\ge-\frac{C_\delta^2}{(Ar_0^2)^2}-\delta(\psi
S^2),\quad \delta\;\mbox{is a small positive constant}.
\]
At last we come to
\begin{equation}
\frac{d}{dt}\bigg|_{t=t_0}u_{\min}\ge\frac{1}{m}u_{\min}^2(t_0)+\bigg(\frac{1}{2m}u_{\min}^2(t_0)
-\frac{\bar{C}^2}{(Ar_0^2)^2}\bigg).
\end{equation}
When $\bar{x}\in B_{t_0}(x_0,r_0)$, since $\psi'=0$, the above inequality still
holds.

Integrating (4.8) yields
\[
u_{\min}(t)\ge\min\bigg\{-\frac{m}{t+K_A^{-1}},-\frac{C}{Ar_0^2}\bigg\}\quad
\mbox{on}\;B_t\bigg(x_0,\frac{3A}{4}r_0\bigg),
\]
where $C$ is a positive constant depending only on $m$. So Claim (4.6) follows.

Step 2: Now if our solution $(g(t),\phi(t))$ is ancient, we can replace $t$ by
$t-\beta$ in (4.6) and get
\[
S(\cdot,t)\ge\min\bigg\{-\frac{m}{t-\beta+K_A^{-1}},-\frac{C}{Ar_0^2}\bigg\}\quad
\mbox{on}\;B_t\bigg(x_0,\frac{3A}{4}r_0\bigg).
\]
Letting $A\to\infty$ and then $\beta\to-\infty$, we complete the proof of the
theorem.
\end{proof}

\begin{remark}
By the proof of Theorem 4.12 and Theorem
4.14, provided $(g(t),\phi(t))$ exists in a time interval which goes to $-\infty$,
the conclusions hold. Therefore, for eternal solutions of Ricci harmonic flow,
we still have the scalar curvature is nonnegative.
\end{remark}

\medskip

\textit{E-mail address:} \texttt{shi.pe@husky.neu.edu}

\end{document}